\documentclass[a4paper,11pt]{amsart}
\usepackage{amsfonts}
\usepackage{amsthm}
\usepackage{amssymb}
\usepackage{amsmath}
\usepackage{comment}
\usepackage{a4wide}
\usepackage{dsfont}
\usepackage{bbm}
\usepackage{xparse,mathrsfs,mathtools}
\usepackage{enumitem}
\usepackage{color}
\usepackage{graphicx}
\usepackage{wrapfig}
\usepackage{hyperref}
\usepackage{mathtools}

\newtheorem{lemma}{Lemma}
\newtheorem{prop}{Proposition}
\newtheorem{remark}{Remark}

\allowdisplaybreaks

\begin{document}

\title{On the pair correlation statistics for determinantal point processes on the sphere}

\author{Maryna Manskova}

\address{Graz University of Technology, Institute of Analysis and Number Theory, Steyrergasse 30/II, 8010 Graz, Austria}
\email{maryna.manskova@math.tugraz.at}


\begin{abstract} 
In this paper, we study the expected value of the pair correlation statistics of randomized point configurations on the sphere, with the emphasis on point configurations generated by determinantal point processes. We study the cases of the spherical ensemble, the harmonic ensemble, and jittered sampling, and compare our results with those for the ``truly random'' (i.i.d.) case. Our results give evidence of the small-scale repulsion phenomenon which is characteristic for determinantal point processes, while on larger scales there is good agreement between all our studied cases and the i.i.d.\  case.
\end{abstract}

\maketitle

\section{Introduction}
Let $\{x_i\}_{i=1}^\infty$ be a sequence in $[0,1]$. The \textit{pair correlation function} of $\{x_i\}_{i=1}^N$ is defined as
\begin{align*}
    R(s,N)=\frac{1}{N}\#\left\{(i,j):1\leq i,j\leq N,\ i\neq j,\ \|x_i-x_j\|\leq\frac{s}{N}\right\},
\end{align*}
where $\|\cdot\|$ denotes the distance to the nearest integer and $s\geq 0$ (see e.g. \cite{Rudnick1998}). The limiting function  of $R(s,N)$, if it exists, describes the asymptotic distribution of small-scale gaps between elements of the sequence. The normalisation of distances is $1/N$, which corresponds to the average spacing between $x_1,\dots,x_N$. If $\{X_i\}_{i=1}^\infty$ is a sequence of independent, uniformly distributed random variables in $[0,1]$, the following convergence holds almost surely
\begin{align}\label{Poissonian_pair_correlation_[0,1]}
   \lim\limits_{N\rightarrow\infty}R(s,N)=2s
\end{align}for every $s\geq 0$. We say that a sequence $\{x_i\}_{i=1}^\infty$ has \textit{Poissonian pair correlations} if it satisfies (\ref{Poissonian_pair_correlation_[0,1]}). 
This notion can be extended to deterministic point sets in unit cubes $[0,1]^d$ \cite{Bera2024,Hinrichs2019,Steinerberger2020} and on compact Riemannian manifolds \cite{Marklof2019}. Our particular interest in this paper are points on spheres. In the $d$-dimensional case, we study a local statistic of point sets defined by
\begin{align}\label{1/N indeces}
    \frac{1}{N}\#\left\{i\neq j\in\{1,\dots,N\}\colon d(x_i,x_j)\leq \frac{s}{N^{1/d}}\right\}
\end{align}
for a distance function $d(x,y)$. For i.i.d.\  random points the limit for $N\rightarrow\infty$ equals $C_ds^d$ for a suitable constant $C_d$ almost surely (see Proposition \ref{iid proposition}). Accordingly, deterministic point sets are said to have Poissonian pair correlations if they share this asymptotic behaviour. 

We study the notion of pair correlations for different samples of probabilistic point sets on the sphere.  Let  $(x_1, \dots , x_N)$ be a sample of points in $\mathbb{S}^d$ equipped with a distance function $\delta(x,y)$ (e.g. chordal or geodesic distance). We rewrite (\ref{1/N indeces}) as
$$G_{s,N}(x_1,\dots,x_N)\coloneqq\frac{1}{N}\sum\limits_{1\leq i\neq j \leq N}F_{s,N}(x_i,x_j),$$
where 
$$F_{s,N}(x,y)=\mathbbm{1}_{[0,s]}\left(\delta(x,y)N^{1/d}\right).$$
In the following chapters, we compute  $\mathbb{E}\left[G_{s,N}\right]$ as $N\rightarrow\infty$ for different determinantal point processes on spheres (and projective spaces). A characteristic property of these processes is a repulsive correlation structure. For small values of $s$, the asymptotic behaviour of $\lim\limits_{N\rightarrow\infty}\mathbb{E}\left[G_{s,N}\right]$ provides information on the degree of repulsion between the points, which turns out to be significantly different from $C_ds^d$ as in the i.i.d.\ case. We recall the basic definitions of determinantal point processes, study the behaviour of $\lim\limits_{N\rightarrow\infty}\mathbb{E}\left[G_{s,N}\right]$ for some particular examples (spherical ensemble, harmonic ensemble, and jittered sampling), and compare our findings to the case of i.i.d.\ points on the sphere.

\section{Preliminaries and Notation}

This section contains some preparatory calculations concerning the geometry of hyperspheres. Here and throughout this paper we denote by $\mathbb{S}^d$ the unit sphere in $\mathbb{R}^{d+1}$ and by $\sigma =\sigma_d$ the normalised surface area measure on $\mathbb{S}^d$. The geodesic distance $\vartheta=\vartheta_d$ takes values in $[0,\pi]$. Note that $\vartheta$ is equivalent to the Euclidean distance, since
$$\|\mathbf{x}-\mathbf{y}\|=\sqrt{2(1-\cos\vartheta(\mathbf{x},\mathbf{y}))}=2\sin\frac{\vartheta(\mathbf{x},\mathbf{y})}{2}.$$ We will use the notation
$$C_d(\mathbf{x},\varphi)=\{\mathbf{y}\in\mathbb{S}^d\, |\, \vartheta(\mathbf{x},\mathbf{y})<\varphi\}$$
for the hyperspherical cap with center $\mathbf{x}$ and of angular radius $\varphi$. 

Recall that the surface area of $\mathbb{S}^d$ is $$\omega_d=2\frac{\pi^{\frac{d+1}{2}}}{\Gamma\left(\frac{d+1}{2}\right)},$$ 
where $\Gamma(x)$ is the gamma function. The normalised surface area of a cap $C_d(\varphi)$ of angular radius $\varphi$ is
$$\sigma_d(C(\varphi))=\frac{\omega_{d-1}}{\omega_d}\int\limits_{0}^\varphi(\sin\theta)^{d-1} \mathrm{d}\theta.$$ 
For $d=2$, the formula is 
$$\sigma_2(C(\varphi))=\frac{1}{2}(1-\cos\varphi)=\left(\sin\frac{\varphi}{2}\right)^2.$$

We introduce the notion
\begin{align*}
    S(\rho,\tau)=\sigma\left(C(\mathbf{x}_1,\pi/2)\cap C(\mathbf{x}_2,\rho)\right)
\end{align*}
for the normalised area of the intersection of a hemisphere of $\mathbb{S}^2$ and a spherical cap of angular radius $\rho$, where $\vartheta(\mathbf{x}_1,\mathbf{x}_2)=\pi/2+\tau$. Note that $\tau$ is the distance from the centre of the cap to the boundary of the hemisphere. One can use a concise formula for the surface area of the intersection of two spherical caps presented in \cite{Lee_Kim}. In the notation of \cite{Lee_Kim}, we have $n=3$, $\theta_1=\pi/2$, $\theta_2=\rho$, and $\theta_v=\pi/2+\tau$. This gives us 
    \begin{equation*}
         S(\rho,\tau)=\int\limits_{\tau}^\rho\sin\varphi\cdot 2\arcsin\sqrt{1-\left(\frac{\tan \tau}{\tan \varphi}\right)^2}\, \mathrm{d}\varphi. 
    \end{equation*}
We provide a geometric proof for the formula.
\begin{lemma}\label{geometric_proof_formula}
    For $0<\tau<\rho<\pi/2$ we have
    \begin{equation*}
       S(\rho,\tau)=\frac{1}{4\pi}\left(\pi-\alpha(\tau)\cos\rho-2\beta(\tau)\right),
    \end{equation*}
    where \begin{align}\label{alpha and beta}
    \beta(\tau) =\arcsin\left(\frac{\sin \tau}{\sin \rho}\right)\quad \text{and}\quad
    \alpha(\tau) =2\arctan\left( \frac{\cot \beta}{\cos\rho}\right).
\end{align}
\end{lemma}

\begin{proof}
    Let us fix the South hemisphere. We can take any point $\mathbf{x}(1,\theta,\pi/2-\tau)$ as the center of the second cap. For $\tau>\rho$, the cap $C\left(\mathbf{x},\rho\right)$ does not intersect the equator. Hence $S(\rho,\tau)=0$. For $0<\tau<\rho<\pi/2$, we get two points of intersection denoted by $\mathbf{y}$ and $\mathbf{z}$. The area $S(\rho,\tau)$ can be expressed as
       $$S(\rho,\tau)=S^{(1)}(\rho,\tau)-S^{(2)}(\rho,\tau),$$
where $S^{(1)}(\rho,\tau)$ is the area of the sector of $C(\mathbf{x},\rho)$ between the points $\mathbf{y}$ and $\mathbf{z}$, and $S^{(2)}(\rho,\tau)$ is the area of the spherical triangle $\triangle \mathbf{xyz}$. 

We denote the plane angles of the triangle by $\alpha=\angle \mathbf{yxz}$ and $\beta=\angle \mathbf{xyz}=\angle \mathbf{xzy}$. The normalised areas of the sector and the triangle are given by
 $$S^{(1)}(\rho,\tau)=\frac{\alpha}{4\pi}(1-\cos \rho)\quad \text{and}\quad S^{(2)}(\rho,\tau)=\frac{1}{4\pi}(\alpha+2\beta-\pi).$$
 This gives us 
 $$S(\rho,\tau)=\frac{1}{4\pi}(\pi-\alpha\cos\rho-2\beta).$$

In order to find $\alpha$ and $\beta$, we denote by $\mathbf{w}$ the middle point between $\mathbf{y}$ and $\mathbf{z}$ on the equator. By the symmetry w.r.t.\ the great circle passing through $\mathbf{x}$, $\mathbf{w}$ and the South pole, we have $\angle \mathbf{xwz}=\angle \mathbf{xwy}=\pi/2$. The angular length of the arc $\mathbf{xw}$ is $\tau$. We obtain formulas for $\alpha$ and $\beta$ by applying the laws of cosine and sine to $\triangle \mathbf{xwz}$, 
$$\cos\frac{\pi}{2}=-\cos\frac{\alpha}{2}\cos\beta+\sin\frac{\alpha}{2}\sin\beta\cos\rho\quad \text{and}\quad \frac{\sin\beta}{\sin \tau}=\frac{\sin\frac{\pi}{2}}{\sin\rho}.$$
This establishes (\ref{alpha and beta}) and completes the proof.
\end{proof}

Another result we will use in this paper is a particular case of the Funk-Hecke formula (see \cite{Mller1966}).
\begin{lemma}\label{Funk-Hecke}
For all bounded measurable $f\colon [-1,1]\rightarrow\mathbb{R}$ and all $\mathbf{y}\in\mathbb{S}^d$,
\begin{align*}
    \int_{\mathbb{S}^d}f(\langle \mathbf{x},\mathbf{y}\rangle)\mathrm{d}\sigma_d(\mathbf{x})=\frac{\omega_{d-1}}{\omega_d}\int_{-1}^1 f(t) (1-t^2)^{d/2-1} \mathrm{d}t.
\end{align*}
\end{lemma}

In the case of i.i.d.\ random points in $\mathbb{S}^d$, we compute the value of the constant for the asymptotic of (\ref{1/N indeces}).
\begin{prop}\label{iid proposition}
 For points $\{\mathbf{x}_1,\dots,\mathbf{x}_N\}$ taken independently from the uniform distribution on $\mathbb{S}^d$, we have 
\begin{equation}\label{iid asymptotic}
     \lim\limits_{N\rightarrow\infty}\mathbb{E}[G_{s,N}(\mathbf{x}_1,\dots,\mathbf{x}_N)]=\frac{\omega_{d-1}}{\omega_d}\cdot\frac{s^d}{d}.
\end{equation}
\end{prop}

\begin{proof}
We have
\begin{align*}    \lim\limits_{N\rightarrow\infty}\mathbb{E}[G_{s,N}]  &=\lim\limits_{N\rightarrow\infty}\frac{1}{N}\cdot N(N-1)\cdot \mathbb{E}\left[\mathbbm{1}_{[0,s]}\left(\vartheta(\mathbf{x},\mathbf{y})N^{1/d}\right)\right]
    \\  &=\lim\limits_{N\rightarrow\infty}(N-1)\cdot \sigma_d(C(sN^{-1/d}))
    \\ &=\frac{\omega_{d-1}}{\omega_d}\cdot\frac{s^d}{d}.
\end{align*}
\end{proof}

Note that for $\mathbb{S}^2$ we obtain $\lim\limits_{N\rightarrow\infty}\mathbb{E}[G_{s,N}] =\frac{s^2}{4}$. Our goal is to compare the case of i.i.d.\ random points (\ref{iid asymptotic}) to other probabilistic point sets obtained from determinantal point processes (such as spherical ensemble, harmonic ensemble, and jittered sampling), with a particular emphasis of level repulsion phenomena, which correspond to very small values of $s$.

\section{Determinantal point process}

 In this section, we recall some basic theory of determinantal point processes. The best general reference here is \cite{Hough2009-xh}.
 
 Let $(\Lambda,\mathcal{B}, \mu)$ be a measurable space, where $\Lambda$ is a locally compact Polish space, $\mathcal{B}$ is the Borel sigma-algebra, and $\mu$ is a Radon measure on $\Lambda$. A sample of $N$ distinct points $\{x_1,\dots,x_N\}$ in $\Lambda$ can be seen as a simple counting measure $\chi =\sum_i \delta_{x_i}$. A simple random point process on $\Lambda$ is a random simple counting measure on $\Lambda$.

Let $\mathcal{X}$ be a simple point process on $\Lambda$. For $1\leq k\leq N$, a $k$-point correlation function $\rho_k\colon \Lambda^k\rightarrow [0,+\infty)$ is a function such that for any family of disjoint sets $A_1,\dots , A_k\in\mathcal{B}$, we have 
$$\mathbb{E}\left[\prod\limits_{j=1}^k \mathcal{X}(A_j) \right]=\int\limits_{A_1}\cdots \int\limits_{A_k}\rho_k(x_1,\dots,x_k)\mathrm{d}\mu(x_1)\cdots \mathrm{d}\mu(x_k).$$
We shall also require that $\rho_k=0$ if $x_i=x_j$ for some $i\neq j$. 

The point process $\mathcal{X}$ with correlation functions $\rho_k$ is called \textit{determinantal} if there exists a kernel $K\colon \Lambda^2\rightarrow \mathbb{C}$ such that
$$\rho_k(x_1,\dots,x_k)=\det\left[K(x_i,x_j)\right]_{1\leq i,j\leq k},$$
for all $k\geq 1$ and $x_1,\dots , x_k\in \Lambda$. Moreover, the joint intensities $\rho_k$ satisfy:
$$\mathbb{E} \sum\limits_{x_1,\dots,x_k\in\Lambda}f(x_1,\dots,x_k)=\int f(x_1,\dots,x_k)\rho_k(x_1,\dots,x_k)$$
for $x = (x_1, \dots , x_k)$  generated by the associated point process and for any measurable function $f$.

One way of generating a determinantal point process is via integral operators. Let $\{\varphi_i\}_{i=1}^N$ be an orthonormal set of functions in $L^2(\Lambda,\mu)$. Then the projection kernel 
$$K(x,y)=\sum\limits_{i=1}^N  \varphi_i(x)\overline{\varphi_i(y)}$$
defines a determinantal point process that has $N$ points almost surely.

\section{Spherical ensemble}

The point pair statistics of the spherical ensemble was studied in \cite{Alishahi2015}. In this section, we review the results of Alishahi and Zamani in our setting.

Let $A, B$ be $N\times N$ random matrices with i.i.d.\ complex Gaussian entries. Then 
the eigenvalues of  $A^{-1}B$ form a determinantal process in the complex plane. 
The \textit{spherical ensemble} of $N$ points is obtained by stereographically
projecting the eigenvalues, which defines a determinantal point process in the unit sphere $(\mathbb{S}^2,\sigma)$ with the $2$-point correlation function $\rho_2^{(N)}$  given by
$$\rho_2^{(N)}(p,q)=N^2\cdot \left(1-\left(1-\frac{\|p-q\|^2}{4}\right)^{N-1}\right).$$
\begin{prop}\label{spherical proposition}
 For points $\mathbf{x}_1,\dots,\mathbf{x}_N$ in $\mathbb{S}^2$ generated by the spherical ensemble, we have
\begin{equation*}
     \lim\limits_{N\rightarrow\infty}\mathbb{E}[G_{s,N}(\mathbf{x}_1,\dots,\mathbf{x}_N)]=\frac{s^2}{4}-1+e^{-s^2/4}.
\end{equation*}
As a consequence, we get
\begin{equation*}
     \lim\limits_{N\rightarrow\infty}\mathbb{E}[G_{s,N}]=\frac{s^4}{32}+O(s^6)\quad\text{as}\quad s\rightarrow 0
\end{equation*}
and 
\begin{equation*}
     \lim\limits_{N\rightarrow\infty}\mathbb{E}[G_{s,N}]= \frac{s^2}{4}+O(s^{-2})\quad\text{as}\quad s\rightarrow\infty.
\end{equation*}
\end{prop}
\begin{proof}
For the Euclidean distance on $\mathbb{S}^2$, we have
\begin{align*}
    \mathbb{E}[G_{s,N}]&=\frac{1}{N}\iint\limits_{p,q\in\mathbb{S}^2}F_{s,N}(p,q)\rho_2^{(N)}(p,q)\mathrm{d}\sigma(p)\mathrm{d}\sigma(q)
    \\ &=\frac{1}{N}\int\limits_{p\in\mathbb{S}^2}F_{s,N}(p,q_0)\rho_2^{(N)}(p,q_0)\mathrm{d}\sigma(p)
    \\ &=\frac{1}{N}\int\limits_{p\in C(q_0,r_{s,N})}\rho_2^{(N)}(p,q_0)\mathrm{d}\sigma(p)
    \\ &=\frac{s^2}{4}- \left(1-\left(1-\frac{s^2}{4N}\right)^{N}\right),
\end{align*}
where $r_{s,N}=\arccos(1-\frac{s^2}{2N})$. As $N$ tends to infinity, we obtain
$$\lim\limits_{N\rightarrow\infty}  \mathbb{E}[G_{s,N}]=\frac{s^2}{4}-1+e^{-s^2/4}.$$
\end{proof}
 In comparison with the corresponding result for the i.i.d.\ model (\ref{iid asymptotic}), the spherical ensemble exhibits ``repulsion'' for small values of $s$, which is expected for a determinantal point process. For large $s$, the asymptotic behaviour of $ \mathbb{E}[G_{s,N}]$ closely matches the asymptotic behaviour in the i.i.d.\ case.

One can try to do similar computations for higher dimensional spheres $\mathbb{S}^d$. There exists a generalisation for even and odd dimensions (see \cite{Beltrn2019} and  \cite{Beltrn2018} respectively). For $k>1$ the kernel of the $2k$-dimensional case is still homogeneous, but not isotropic (i.e. not invariant under rotations). This makes computations of $\mathbb{E}[G_{s,N}]$ more complicated.

\section{Harmonic ensemble}

The separation distance of points generated by the harmonic ensemble was studied in \cite{Beltrn2016}. An upper bound for $\mathbb{E}[G_{s,N}]$ was given by Beltrán, Marzo, and Ortega-Cerdá in \cite[Proposition 4]{Beltrn2016}. We will calculate the exact asymptotic for large and small values of $s$.

For a natural number $L$, let us consider the vector space of spherical harmonics of degree at most $L$ in $\mathbb{S}^d$. We fix an orthonormal basis with respect to the norm $L^2(\mathbb{S}^d,\sigma)$. The dimension of the space is
$$N=\frac{2L+d}{d}\binom{d+L-1}{L}.$$
The \textit{harmonic ensemble} is the determinantal point process in $(\mathbb{S}^d,\sigma)$ with
$N$ points almost surely induced by the reproducing kernel
$$K_L(\mathbf{x},\mathbf{y})=\frac{N}{P_L^{(1+\lambda,\lambda)}(1)}P_L^{(1+\lambda,\lambda)}(\langle \mathbf{x},\mathbf{y}\rangle),$$
where $\lambda=\frac{d-2}{2}$ and the Jacobi polynomials are normalized by $P_L^{(1+\lambda,\lambda)}(1)=\binom{L+\frac{d}{2}}{L}$. Recall that 
\begin{align}\label{binom}
    \binom{n+x}{n}=\frac{n^x}{\Gamma(x+1)}\left(1+O\left(\frac{1}{n}\right)\right)\quad \text{as}\ n\rightarrow \infty.
\end{align}  

Note that the number of points $N$ generated by the harmonic ensemble is
\begin{align}\label{NL relation}
N=\frac{2}{\Gamma(d+1)}L^d+o(L^d).    
\end{align}
Throughout, we consider $N\rightarrow\infty$ or $L\rightarrow\infty$ when convenient.
\begin{prop}\label{harmonic proposition}
 For points $\mathbf{x}_1,\dots,\mathbf{x}_N$ in $\mathbb{S}^d$ generated by the harmonic ensemble, we have
 \begin{equation*}
     \lim\limits_{N\rightarrow\infty}\mathbb{E}[G_{s,N}(\mathbf{x}_1,\dots,\mathbf{x}_N)]=\frac{\omega_{d-1}}{\omega_d} \left(\frac{\Gamma(d+1)}{2}\right)^{2/d}\frac{s^{d+2}}{(d+2)^2}+O(s^{d+4})
\end{equation*}
as $s\rightarrow 0$ and 
\begin{equation*}
     \lim\limits_{N\rightarrow\infty}\mathbb{E}[G_{s,N}(\mathbf{x}_1,\dots,\mathbf{x}_N)]=\frac{\omega_{d-1}}{\omega_d}\cdot\frac{s^d}{d}+O(1)\quad\text{as}\quad s\rightarrow\infty.
\end{equation*}

\end{prop}
For the geodesic distance, let us compute 
\begin{align*}
    \lim\limits_{N\rightarrow\infty}\mathbb{E}[G_{s,N}]&=\lim\limits_{N\rightarrow\infty}\frac{1}{N}\iint\limits_{p,q\in\mathbb{S}^d}F_{s,N}(p,q)\rho_2^{(N)}(p,q)\mathrm{d}\sigma(p)\mathrm{d}\sigma(q)
    \\ &=\lim\limits_{N\rightarrow\infty}N\frac{\omega_{d-1}}{\omega_d}\int\limits_{\cos\varphi_{s,N}}^1\left(1-\frac{P_L^{(1+\lambda,\lambda)}(t)^2}{\binom{L+d/2}{L}^2}\right)(1-t^2)^{d/2-1}\mathrm{d}t
    \\ &=\lim\limits_{N\rightarrow\infty}N\frac{\omega_{d-1}}{\omega_d}\int\limits_0^{\varphi_{s,N}}\left(1-\frac{P_L^{(1+\lambda,\lambda)}(\cos \varphi )^2}{\binom{L+d/2}{L}^2}\right)(\sin\varphi)^{d-1}\mathrm{d}\varphi
    \\ &=\lim\limits_{N\rightarrow\infty} N\frac{\omega_{d-1}}{\omega_d}\int\limits_0^{\varphi_{s,N} L}\left(1-\frac{P_L^{(1+\lambda,\lambda)}\left(\cos \frac{y}{L} \right)^2}{\binom{L+d/2}{L}^2}\right)\left(\frac{y}{L}\right)^{d-1}\frac{\mathrm{d}y}{L}
    \\ &=\lim\limits_{N\rightarrow\infty}\Bigg[\frac{N}{L^d }\frac{\omega_{d-1}}{\omega_d}\int\limits_0^{\varphi_{s,N} L}y^{d-1}\mathrm{d}y
    \\ & \qquad \qquad \quad -\frac{N}{\binom{L+d/2}{L}^2}\frac{\omega_{d-1}}{\omega_d}\cdot\int\limits_0^{\varphi_{s,N} L}\left(\frac{P_L^{(1+\lambda,\lambda)}\left(\cos \frac{y}{L} \right)}{L^{d/2}}\right)^2 y^{d-1}\mathrm{d}y\Bigg]
    \\ &=\lim\limits_{N\rightarrow\infty}\left[I^{(1)}_{s,N}-I^{(2)}_{s,N}\right].
\end{align*}
where $\varphi_{s,N}=\frac{s}{N^{1/d}}$ is small for large $N$ and fixed $s$. We used the fact that $\sin\varphi \approx \varphi$ for small values of $\varphi$. For the first integral we have
$$ \lim\limits_{N\rightarrow\infty}I^{(1)}_{s,N}=\lim\limits_{N\rightarrow\infty}\frac{N}{L^d }\frac{\omega_{d-1}}{\omega_d}\int\limits_0^{\varphi_{s,N} L}y^{d-1}\mathrm{d}y= \lim\limits_{N\rightarrow\infty}\frac{N}{L^d }\frac{\omega_{d-1}}{\omega_d}\cdot \frac{1}{d}\left(\varphi_{s,N}L\right)^d= \frac{\omega_{d-1}}{\omega_d}\cdot\frac{s^d}{d}.$$

For the second integral, we need the following formula.
\begin{lemma}[Mehler-Heine formula]\label{Mehler-Heine formula}
    Let $\alpha,\beta$ be real numbers. Then, uniformly on compact subsets of $\mathbb{C}$
    $$\lim\limits_{n\rightarrow\infty}n^{-\alpha}P^{(\alpha,\beta)}_n\left(\cos\frac{z}{n}\right)=\left(\frac{z}{2}\right)^{-\alpha}J_\alpha(z),$$
    where $J_\alpha$ is the Bessel function of order $\alpha$.
\end{lemma}

Note that there exists $c_d>0$ such that 
$$\varphi_{s,N}L=\frac{s}{N^{1/d}}\cdot L\leq c_d s.$$
Applying Lemma \ref{Mehler-Heine formula} for $\alpha=1+\lambda=\frac{d}{2}$, $\beta=\lambda$, $n=L$ and the compact set $\left\{z\in\mathbb{C}:|z|\leq c_d s\right\}$, we get
\begin{align*}
   \lim\limits_{L\rightarrow\infty}\int\limits_0^{\varphi_{s,N} L}\left(\frac{P_L^{(1+\lambda,\lambda)}\left(\cos \frac{y}{L} \right)}{L^{d/2}}\right)^2 y^{d-1}\mathrm{d}y
    &= \lim\limits_{L\rightarrow\infty} \int\limits_0^{\varphi_{s,N} L}\left(\left(\frac{y}{2}\right)^{-d/2}J_{d/2}(y)\right)^2 y^{d-1}\mathrm{d}y 
    \\ &= \lim\limits_{L\rightarrow\infty}2^d\int\limits_0^{\varphi_{s,N} L}\frac{1}{y}J_{d/2}(y)^2\mathrm{d}y.
\end{align*}

\begin{lemma}\label{Integral expansion}
   For $\nu$ positive and $x$ small,
\begin{align*}
    \int\limits_0^{x} \frac{1}{t}J_{\nu}(t)^2\mathrm{d}t=x^{2\nu}\frac{1}{ 2^{2\nu+1}\Gamma(\nu+1)^2\nu}-x^{2\nu+2}\frac{1}{2^{2\nu+2}\Gamma(\nu+2)^2}+O(x^{2\nu+4}).
\end{align*}   

\end{lemma}
\begin{proof}
The asymptotics comes from the Maclaurin series of Bessel functions,
$$J_{\nu}(t)=\sum\limits_{m=0}^\infty \frac{(-1)^m}{m! \Gamma(m+\nu+1)}\left(\frac{t}{2}\right)^{2m+\nu}.$$
We have 
$$\frac{1}{t} J_{\nu}(t)^2=\frac{1}{t}\sum\limits_{m=0}^\infty \frac{(-1)^m (m+2\nu+1)_m}{m! \Gamma(m+\nu+1)^2}\left(\frac{t}{2}\right)^{2m+2\nu}.$$
This gives us 
\begin{align*}
     \int\limits_0^{x} \frac{1}{t}J_{\nu}(t)^2\mathrm{d}t=\sum\limits_{m=0}^\infty \frac{(-1)^m (m+2\nu+1)_m}{m! \Gamma(m+\nu+1)^2(2m+2\nu)}\left(\frac{x}{2}\right)^{2m+2\nu}.
\end{align*}
Taking the first terms for $m=0$ and $m=1$, we obtain the desired result.
\end{proof}
\quad
\\
Using (\ref{binom}), (\ref{NL relation}), and the result of Lemma \ref{Integral expansion} for $\nu=\frac{d}{2}$ and $s$ small, we get that the second integral $I^{(2)}_{s,N}$ is equal to
\begin{align*}
   &\lim\limits_{L\rightarrow\infty}\frac{N}{\binom{L+d/2}{L}^2}\cdot\frac{\omega_{d-1}}{\omega_d}\cdot2^d\left(\frac{(\varphi_{s,N}L)^d}{ 2^{d}\Gamma(\frac{d}{2}+1)^2 d}-\frac{(\varphi_{s,N}L)^{d+2}}{2^{d+2}\Gamma(\frac{d}{2}+2)^2}+O(s^{d+4})\right)
    \\ &=\lim\limits_{L\rightarrow\infty} \frac{N}{\binom{L+d/2}{L}^2}\cdot\frac{\omega_{d-1}}{\omega_d}\cdot\left(\frac{L^d}{N \Gamma(\frac{d}{2}+1)^2 d}s^d-\frac{L^{d+2}}{4N^{1+2/d}\Gamma(\frac{d}{2}+2)^2}s^{d+2}+O(s^{d+4})\right)\\
    &=\frac{\omega_{d-1}}{\omega_d}\cdot\frac{s^d}{d}-\frac{\omega_{d-1}}{\omega_d} \left(\frac{\Gamma(d+1)}{2}\right)^{2/d}\frac{s^{d+2}}{(d+2)^2}+O(s^{d+4}).
\end{align*}
In this way we obtain that
$$\lim\limits_{N\rightarrow \infty}\mathbb{E}[G_{s,N}]=\frac{\omega_{d-1}}{\omega_d} \left(\frac{\Gamma(d+1)}{2}\right)^{2/d}\frac{s^{d+2}}{(d+2)^2}+O(s^{d+4}),\quad \text{as}\quad s\rightarrow 0.$$
Note that for $d=2$ it is $\frac{s^4}{32}+O(s^6)$.

For large values of $s$, we have 
\begin{align*}
    0\leq\int\limits_0^{\varphi_{s,N}L}\frac{1}{y}J_{d/2}(y)^2\mathrm{d}y&\leq \int\limits_0^{\infty}\frac{1}{y}J_{d/2}(y)^2\mathrm{d}y=\frac{1}{d}. \label{a_d}
\end{align*}

Hence, $0\leq I^{(2)}_{s,N}\leq C_d$, which implies that the pair correlation function has the following asymptotic behaviour,
$$\lim\limits_{N\rightarrow \infty}\mathbb{E}[G_{s,N}]=\frac{\omega_{d-1}}{\omega_d}\cdot\frac{s^d}{d}+O(1), \quad s\rightarrow \infty.$$
This completes the proof of Proposition \ref{harmonic proposition}.

In the setting of the harmonic ensemble, we again observe ``repulsion'' for small $s$. Note that for $d=2$, the asymptotic behaviour coincides with that of the spherical ensemble. For large $s$, it has the same asymptotic behaviour as the i.i.d.\ case, but the error term is worse than the one we obtained for the spherical ensemble on $\mathbb{S}^2$.

Similarly, one can define harmonic ensembles  for projective spaces $\mathbb{RP}^d$,  $\mathbb{CP}^d$,  and $\mathbb{HP}^d$. See \cite{Anderson2023}, \cite{Brauchart2024}, and \cite{Bilyk2024} for more details. For a projective space $\mathbb{FP}^d$, we define 
$$\alpha =\frac{d}{2}\cdot \dim_{\mathbb{R}}\mathbb{F}-1\quad \text{and}\quad \beta=\frac{1}{2}\dim_{\mathbb{R}}\mathbb{F}-1.$$
As a real manifold, the space has dimension $D=\dim_{\mathbb{R}}\mathbb{FP}^d=d\cdot \dim_{\mathbb{R}}\mathbb{F}$. Note that $\alpha=\frac{D}{2}-1$. We will denote by $\vartheta$ the geodesic distance, $\vartheta(\mathbf{x},\mathbf{y})\in[0,\frac{\pi}{2}]$. The chordal distance  is given by
$$\rho(\mathbf{x},\mathbf{y})=\sqrt{\frac{1-\cos(2\vartheta(\mathbf{x},\mathbf{y}))}{2}}=\sin(\vartheta(\mathbf{x},\mathbf{y})).$$
It is easy to see that these two distances are equivalent, we are interested in small values of $0\leq\vartheta\leq\frac{s}{N^{1/D}}$. 

We denote by $\sigma$ the uniform probability measure on $\mathbb{FP}^d$. For any $\mathbf{y}\in \mathbb{{FP}}^d$ one has the identity (see e.g. \cite{Brauchart2024})
\begin{equation}\label{zonal integral}
    \int_{\mathbb{FP}^d}f\left(\cos (2\vartheta(\mathbf{x},\mathbf{y}))\right)\mathrm{d}\sigma(\mathbf{x})=C_{\alpha,\beta}\int\limits_0^{\pi/2}f(\cos(2\vartheta))\sin(\vartheta)^{2\alpha+1}\cos (\vartheta)^{2\beta+1}\mathrm{d}\vartheta,
\end{equation}
where the normalising constant is given by 
$$C_{\alpha,\beta}=\frac{2\Gamma(\alpha+\beta+2)}{\Gamma(\alpha+1)\Gamma(\beta+1)}.$$
In the case of i.i.d.\ random points in $\mathbb{FP}^d$, we compute the asymptotic of $\mathbb{E}[G_{s,N}] $ as $N\rightarrow\infty$.
\begin{prop}\label{iid proposition projective}
 For points $\{\mathbf{x}_1,\dots,\mathbf{x}_N\}$ taken independently from the uniform distribution on $\mathbb{FP}^d$, we have 
\begin{equation}\label{iid asymptotic}
     \lim\limits_{N\rightarrow\infty}\mathbb{E}[G_{s,N}(\mathbf{x}_1,\dots,\mathbf{x}_N)]=C_{\alpha,\beta}\cdot\frac{s^D}{D}.
\end{equation}
\end{prop}
\begin{proof}
We have
\begin{align*}
       \lim\limits_{N\rightarrow\infty}\mathbb{E}[G_{s,N}]  &=\lim\limits_{N\rightarrow\infty}\frac{1}{N}\cdot N(N-1)\cdot \mathbb{E}\left[\mathbbm{1}_{[0,s]}\left(\vartheta(\mathbf{x},\mathbf{y})N^{1/D}\right)\right]
       \\ &=\lim_{N\rightarrow\infty}(N-1)\cdot C_{\alpha,\beta}\int\limits_{0}^{sN^{-1/D}}\sin(\vartheta)^{2\alpha+1}\cos (\vartheta)^{2\beta+1}\mathrm{d}\vartheta
    \\ &=C_{\alpha,\beta}\cdot\frac{s^D}{D}.
    \end{align*}
\end{proof}
The next step is to define the harmonic ensemble on $\mathbb{FP}^d$. For a positive integer $L$, we consider the space of eigenvectors associated to the first $L+1$ eigenvalues of the Laplace operator on $L^2(\mathbb{FP}^d,\sigma)$. The dimension of the space is 
$$N=\frac{(\alpha+\beta+2)_L(\alpha+2)_L}{L!(\beta+1)_L}\asymp\frac{\Gamma(\beta+1)}{\Gamma(\alpha+\beta+2)\Gamma(\alpha+2)}L^D=:K_{\alpha,\beta}L^D,$$
where $(x)_n=x\cdot (x+1)\cdots (x+n-1)$.
The corresponding reproducing kernel is given by
$$K_L(\mathbf{x},\mathbf{y})=\frac{N}{P_L^{(\alpha+1,\beta)}(1)}P_L^{(\alpha+1,\beta)}(\cos(2\vartheta(\mathbf{x},\mathbf{y}))).$$
Note that $P_L^{(\alpha+1,\beta)}(1)=\binom{L+\alpha+1}{L}\asymp \frac{L^{D/2}}{\Gamma(\alpha+2)}$ as $L\rightarrow\infty$. The kernel defines a determinantal point process on $\mathbb{FP}^d$ and generates $N$ points almost surely.

\begin{prop}\label{harmonic hyper proposition}
 For points $\mathbf{x}_1,\dots,\mathbf{x}_N$ in $\mathbb{FP}^d$ generated by the harmonic ensemble, we have
 \begin{equation*}
     \lim\limits_{N\rightarrow\infty}\mathbb{E}[G_{s,N}(\mathbf{x}_1,\dots,\mathbf{x}_N)]=4C_{\alpha,\beta}K_{\alpha,\beta}^{-2/D}\cdot \frac{s^{D+2}}{(D+2)^2}+O(s^{D+4})\quad\text{as}\quad s\rightarrow 0
\end{equation*}
and 
\begin{equation*}
     \lim\limits_{N\rightarrow\infty}\mathbb{E}[G_{s,N}(\mathbf{x}_1,\dots,\mathbf{x}_N)]=C_{\alpha,\beta}\cdot\frac{s^D}{D}+O(1)\quad\text{as}\quad s\rightarrow\infty.
\end{equation*}

\end{prop}
\begin{proof}
Here we compute the pair correlation function,
\begin{align*}
     &\lim\limits_{N\rightarrow \infty}  \mathbb{E}[G_{s,N}]=\lim\limits_{N\rightarrow \infty}\frac{1}{N}\iint\limits_{p,q\in\mathbb{FP}^d}F_{s,N}(p,q)\rho_2^{(N)}(p,q)\mathrm{d}\sigma(p)\mathrm{d}\sigma(q).
\end{align*}
Applying (\ref{zonal integral}), we get
\begin{align}
  &\lim\limits_{N\rightarrow \infty}  \mathbb{E}[G_{s,N}] \nonumber
    \\ &=\lim\limits_{N\rightarrow \infty}NC_{\alpha,\beta}\int\limits_{0}^{\varphi_{s,N}}\left(1-\frac{P_L^{(\alpha+1,\beta)}(\cos(2\vartheta))^2}{\binom{L+D/2}{L}^2}\right)\sin(\vartheta)^{2\alpha+1}\cos(\vartheta)^{2\beta+1}\mathrm{d}\vartheta \nonumber
    \\ &=\lim\limits_{N\rightarrow \infty} NC_{\alpha,\beta}\int\limits_0^{2\varphi_{s,N}L}\left(1-\frac{P_L^{(\alpha+1,\beta)}\left(\cos \frac{y}{L} \right)^2}{\binom{L+D/2}{L}^2}\right)\left(\frac{y}{2L}\right)^{D-1}\frac{\mathrm{d}y}{2L} \nonumber
    \\ &=\lim\limits_{N\rightarrow \infty} \frac{N C_{\alpha,\beta}}{(2L)^D }\int\limits_0^{2\varphi_{s,N}L}\left(y^{D-1}-\Gamma\left(\alpha+2\right)^2\cdot\left(\frac{P_L^{(\alpha+1,\beta)}\left(\cos \frac{y}{L} \right)}{L^{D/2}}\right)^2 y^{D-1}\right)\mathrm{d}y \label{line in int1}
    \\ &=\lim\limits_{N\rightarrow \infty} \frac{N C_{\alpha,\beta}}{(2L)^D }\int\limits_0^{2\varphi_{s,N}L}\left(y^{D-1}-2^D\Gamma\left(\alpha+2\right)^2\cdot\frac{1}{y}J_{D/2}(y)^2\right)\mathrm{d}y, \label{line in int2}
    \end{align}
    where $\varphi_{s,N}=\frac{s}{N^{1/D}}$. To deduce (\ref{line in int2}) from (\ref{line in int1}), we use Lemma \ref{Mehler-Heine formula}. In much the same way as in the case of hyperspheres,  we obtain the asymptotic behaviour of $ \lim\limits_{N\rightarrow\infty}\mathbb{E}[G_{s,N}]$ for $s\rightarrow0$ and $s\rightarrow\infty$. 
\end{proof}

\section{Jittered sampling}

Let $(X,\mathcal{B},\mu)$ be a measurable space with $\mu(X)=1$. Consider $\mathcal{A}=\{A_1,A_2,\dots,A_N\}\subset \mathcal{B}$ a \textit{partition} of $X$, i.e. $\mu(A_i)\neq 0$, $\mu(A_i\cap A_j)=0$ for  $i\neq j$, and $\mu\left(\bigcup\limits_{i=1}^N A_i\right)=1$.

The \textit{jittered sampling} process associated to the partition $\mathcal{A}$ is the determinantal point process defined by the kernel
$$K_\mathcal{A}(x,y)=\sum\limits_{i=1}^N \frac{\mathbbm{1}_{A_i}(x)\mathbbm{1}_{A_i}(y)}{\mu(A_i)}.$$

It generates $N$ points almost surely. Moreover, it is easy to deduce that the process samples one point in each set $A_i$ for $i=1,\dots,N$; the distribution measure for the sample point in  $A_i$ is $\frac{\mu}{\mu(A_i)}$. See for instance \cite{Brauchart2020} for more details in the setting of points on the sphere. We note that jittered sampling, also known as stratified sampling, is a classical and heavily used method in numerical analysis, where it is used as a variance reduction method to improve the performance of the standard Monte Carlo method (using i.i.d.\ random points); see for example \cite{Glasserman2003,lemieux2009monte}.

We focus on the case of \textit{equal-area partitions}, i.e. $\mu(A_i)=\frac{1}{N}$ for $i=1,\dots, N$. By \cite[Theorem 2]{Gigante2016}, such a partition exists for an arbitrary connected Ahlfors regular metric measure space $X$; moreover, for every sufficiently large $N$
\begin{align}\label{diam_bounds}
    cN^{-1/d}\leq \operatorname{diam}(A_i)\leq c^\prime N^{-1/d},\quad i=1,\dots,N,
\end{align} where $d$ is the dimension of the space and the constants $c,c^\prime>0$  depend only on $X$.

Let us denote by $e_\mathcal{A}(x,y)$ the function which is equal to $1$ if $x,y\in A_i$ for some $i$, and $0$ otherwise. It is easy to see that $K_\mathcal{A}(x,y)=N\cdot e_\mathcal{A}(x,y)$, and $e_\mathcal{A}(x,x)=1$ for all $x\in X$. The $2$-point correlation function is given by
$$\rho_2(x,y)=K(x,x)K(y,y)-|K(x,y)|^2=N^2-N^2e_\mathcal{A}(x,y)^2=N^2(1-e_\mathcal{A}(x,y)).$$
In particular, when we consider $X=\mathbb{S}^d$ and $\mu=\sigma_d$ equipped with the geodesic distance $\vartheta_d$, the pair correlation function associated to an equal-area partition $\mathcal{A}=\{A_1,A_2,\dots,A_N\}$ is
\begin{align}
    \mathbb{E}[G_{s,N}]&=\frac{1}{N}\iint\limits_{\mathbf{x},\mathbf{y}\in \mathbb{S}^d}F_{s,N}(\mathbf{x},\mathbf{y})\rho_2^{(N)}(\mathbf{x},\mathbf{y})\mathrm{d}\sigma(\mathbf{x})\mathrm{d}\sigma(\mathbf{y}) \nonumber
    \\ &=N \iint\limits_{\mathbf{x},\mathbf{y}\in \mathbb{S}^d}  \mathbbm{1}_{C\left(\mathbf{x},sN^{-1/d}\right)}(\mathbf{y})\cdot (1-e_\mathcal{A}(\mathbf{x},\mathbf{y}))\mathrm{d}\sigma(\mathbf{x})\mathrm{d} \sigma(\mathbf{y}) \nonumber
    \\ &=N \sum\limits_{i=1}^N\int\limits_{\mathbf{x}\in A_i}\sigma\left( C\left(\mathbf{x},sN^{-1/d}\right)\backslash\, A_i\right)\mathrm{d}\sigma(\mathbf{x}) \label{The integral}
    \\ &=N\sigma\left( C\left(sN^{-1/d}\right)\right)-N\sum\limits_{i=1}^N\int\limits_{\mathbf{x}\in A_i}\sigma\left(A_i\cap C\left(\mathbf{x},sN^{-1/d}\right)\right)\mathrm{d}\sigma(\mathbf{x}). \label{NS-NInt}
\end{align}

\begin{prop}\label{jittered large scale proposition}
    Let $\{\mathcal{A}^{(N)}\}_{N=1}^\infty$ be a sequence of equal-area partitions of $(\mathbb{S}^d,\sigma)$. If all $\mathcal{A}^{(N)}$ satisfy the upper bound in (\ref{diam_bounds}) for some $c^\prime>0$, then for $s>c^\prime$ we have
\begin{align*}
    \lim\limits_{N\rightarrow\infty}\mathbb{E}[G_{s,N}(\mathbf{x}_1,\dots,\mathbf{x}_N)]=\frac{\omega_{d-1}}{\omega_d}\cdot\frac{s^d}{d}-1,
\end{align*}
where $\mathbf{x}_1,\dots,\mathbf{x}_N$ are generated independently by the jittered sampling associated to $\mathcal{A}^{(N)}$.
\end{prop}
\begin{proof}
    Note that for $s>c^\prime$ and $\mathbf{x}\in A_i^{(N)}$, a cap $C(\mathbf{x},sN^{-1/d})$ covers the whole set $A_i^{(N)}$. By (\ref{NS-NInt}), we can easily get that
\begin{align*}
    \lim\limits_{N\rightarrow\infty}\mathbb{E}[G_{s,N}]&=\lim\limits_{N\rightarrow\infty}\left[N\sigma\left( C\left(sN^{-1/d}\right)\right)-N\sum\limits_{i=1}^N\int\limits_{\mathbf{x}\in A_i^{(N)}}\sigma\left(A_i^{(N)}\right)\mathrm{d}\sigma(\mathbf{x})\right]\\
    &=\frac{\omega_{d-1}}{\omega_d}\cdot\frac{s^d}{d}-1.
\end{align*}
\end{proof}

The next step is to consider the case of small values of $s$ using (\ref{The integral}). For $A\subset\mathbb{S}^d$, define
\begin{equation}\label{definition_of_M}
    M_\rho(A)=\int\limits_{\mathbf{x}\in A}\sigma\left( C\left(\mathbf{x},\rho\right)\backslash\, A\right)\mathrm{d}\sigma(\mathbf{x}).
\end{equation}
Note that
 \begin{align}
M_\rho(A)&=\int\limits_{\mathbf{x}\in A}\sigma\left(\{ \mathbf{y}\,|\, \mathbf{y}\not\in A, \, \vartheta(\mathbf{x},\mathbf{y})<\rho\}\right)\mathrm{d}\sigma(\mathbf{x}) \nonumber
\\ &=\sigma^2\left( \{(\mathbf{x},\mathbf{y})\,|\, \mathbf{x}\in A, \mathbf{y}\not\in A, \vartheta(\mathbf{x},\mathbf{y})<\rho\}\right) \label{sigma_x_y}
\\ &=\sigma(A)\cdot \left(1-\sigma(A)\right)-\sigma^2\left( \{(\mathbf{x},\mathbf{y})\,|\, \mathbf{x}\in A,\, \mathbf{y}\not\in A,\, \vartheta(\mathbf{x},\mathbf{y})\geq\rho\}\right). \nonumber
 \end{align}

It is easy to see that $M_\rho(A)$ is invariant under taking complements, namely for $A^c=\mathbb{S}^d\backslash A$ we have $M_\rho(A)=M_\rho(A^c)$. Moreover, we observe that all points $\mathbf{x}$ and $\mathbf{y}$ contributing to (\ref{sigma_x_y}) are close to the boundary of $A$ (and $A^c$). 

For simplicity, let us consider the case of $\mathbb{S}^2$. We continue by recalling the notion of a \textit{tube} (see \cite{Gray2004} for more details). Let $\mathcal{C}$ be a closed one-dimenisonal manifold on $\mathbb{S}^2$. Then a tube of radius $\theta>0$ is the set $   T(\mathcal{C},\theta)$ of points $\mathbf{x}\in\mathbb{S}^2$ such that there exist a geodesic of length at most $\theta$ from $\mathbf{x}$ meeting the curve $\mathcal C$ orthogonally. Since $\mathcal{C}$ is closed, for sufficiently small $\theta$, $T(\mathcal{C},\theta)$ coincides with the set of points  $\mathbf{y}\in\mathbb{S}^2$ such that $\operatorname{dist}(\mathbf{y},\mathcal{C})\leq \theta$. 

It was proven by H. Hotelling \cite{Hotelling1939} and generalised by H. Weyl \cite{Weyl1939} that the normalized area of a tube is given by 
\begin{equation*}
    \sigma(T(\mathcal{C},\theta))=\frac{1}{4\pi}\cdot 2 L(\mathcal{C})\sin\theta,
\end{equation*}
where $L(\mathcal{C})$ denotes the length of $\mathcal{C}$. To guarantee the absence of local self-overlapping of the tube (i.e. there are no points in $T(\mathcal{C},\theta)$ that have more than one geodesic to $\mathcal{C}$), we require one of the equivalent conditions
\begin{equation}\label{no overlap condition}
\tan \theta\leq \frac{1}{k_g} \quad\text{or}\quad \sin \theta \leq r,    
\end{equation} where $k_g$ is the geodesic curvature of $\mathcal{C}$ and $r$ is the radius of curvature of $\mathcal{C}$ relative to the Euclidean space $\mathbb{R}^3$.

We also need another notion, the \textit{tubular hypersurface} at a distance $\tau$ from $\mathcal{C}$, which is defined as
\begin{equation*}
    \mathcal{C}_\tau=\{\mathbf{x}\in T(\mathcal{C},\theta)\ |\ \operatorname{dist}(\mathbf{x},\mathcal{C})=\tau\}.
\end{equation*}
\begin{lemma}\label{integral M}
    Let $A\subset\mathbb{S}^2$ such that $\partial A$ is a smooth simple closed curve. Let $\hat{r}\in\mathbb{R}$ such that for $0<\theta<\hat{r}<\pi/2$, there is no overlapping of $T(\partial A,\theta)$ with itself. Then for $0<\rho<\hat{r}/2$ we have
    \begin{equation*}
       \left| M_\rho(A)-L(\partial A)\cdot\frac{\rho^3}{24\pi^2}\right|\leq {L(\partial A)}{ \hat r^{-1}}\rho^4.
    \end{equation*}
\end{lemma}
\begin{proof}
It is clear that the contribution to $M_\rho(A)$ comes from $\mathbf{x}$ close to the boundary. So let us fix $\rho$ in $(0,\hat r/2)$ and for $0<\tau<\rho$ consider points $\mathbf{x}\in A\cap \mathcal{C}_\tau$. There exists a unique geodesic from $\mathbf{x}$ orthogonal to $\partial A$. We denote by $\mathbf{w}$ the point where the geodesic meets the boundary. The idea is to approximate $\partial A$ near $\mathbf{w}$ by the tangent great circle $C(\mathring{\mathbf{x}},\pi/2)$ and the inscribed circle of radius $\hat r$ centered at $\hat{\mathbf{x}}$ (here $\mathbf{x},\hat{\mathbf{x}}\not\in C(\mathring{\mathbf{x}},\pi/2)$). Note that $C(\hat{\mathbf{x}},\hat{r})$ is inside $A$. This will give an approximation of the area of $C\left(\mathbf{x},\rho\right)\backslash\, A$ by $S(\rho,\tau)$.
    \begin{figure}[ht]
    \centering
    \includegraphics[width=0.75\textwidth]{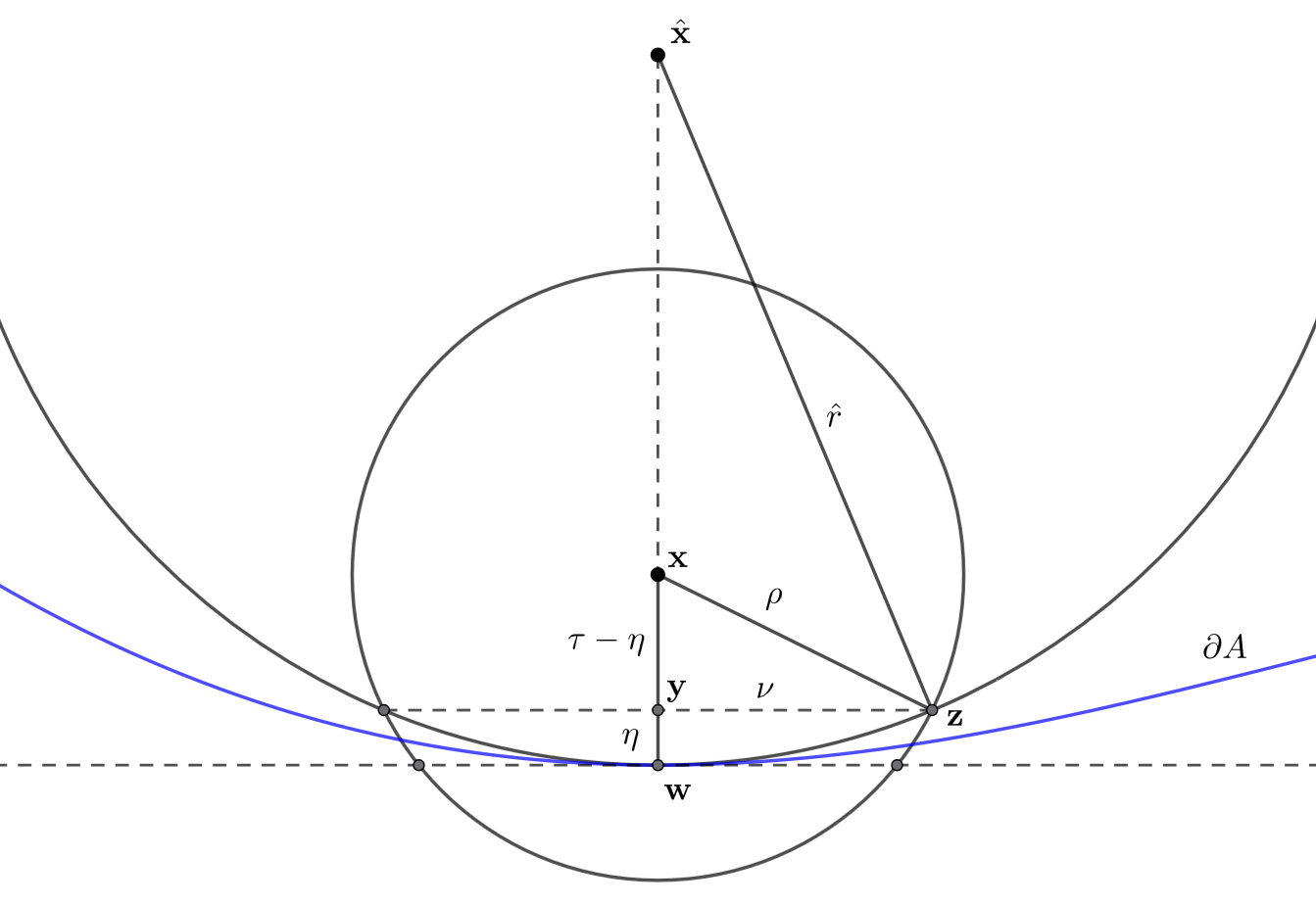}
   \caption{Approximation of $\partial A$ and $C(\mathbf{x},\rho)\backslash  A$.}
    \label{fig:approximation}
\end{figure}

We can bound the difference between $\sigma\left(C(\mathbf{x},\rho)\backslash A\right)$ and $S(\rho,\tau)$ by the area between two great circles (the tangent one and the one passing through the points of intersection of $C(\mathbf{x},\rho)$ and $C(\hat{\mathbf{x}},\hat{r})$) inside $C(\mathbf{x},\rho)$. Namely, we have that
    \begin{equation*}
         \left|\sigma (C\left(\mathbf{x},\rho\right)\backslash\, A)-S(\rho,\tau)\right|\leq 2\rho\cdot \eta,
    \end{equation*}
    where $\eta$ is the maximum distance between the great circles. 
    From the spherical law of cosine for right-angled triangles, we have 
    $$
    \cos \rho=\cos(\tau-\eta)\cos \nu\quad \text{and}\quad \cos \hat{r}=\cos(\hat{r}-\eta)\cos \nu.
    $$
    This implies 
    $$
    \frac{\cos \rho}{\cos\hat{r}}=\frac{\cos\tau\cos\eta+\sin\tau\sin\eta}{\cos\hat{r}\cos\eta+\sin\hat{r}\sin\eta}.
    $$
 We obtain 
    \begin{align*}
        \tan \eta & =\frac{\cos \tau-\cos \rho}{\tan \hat r \cos \rho -\sin \tau}\\
        &\leq\frac{1}{2\cos \rho}\cdot\frac{\rho^2-\tau^2}{\tan\hat r-\frac{\sin \tau}{\cos \rho}}\leq \frac{1}{2\cos \rho}
        \cdot \frac{\rho^2-\tau^2}{\tan \hat r-\tan\rho}\leq\frac{\rho^2-\tau^2}{2\sin(\hat r-\rho)},
    \end{align*}
and hence
\begin{equation*}
    \eta\leq \frac{\rho^2-\tau^2}{\sin \hat r}.
\end{equation*}
    The above holds also for points in $\mathbf{x}\in A^{c}\cap\mathcal{C}_\tau$. Using the symmetry under taking complements, we can work with $M_\rho(A)+M_\rho(A^c)=2M_\rho(A).$ Integrating over sets $\mathcal{C}_\tau$, we obtain the main term
\begin{align*}
    2I_\rho(A)&\coloneqq\int\limits_{\mathbf{x}\in A}\sigma\left( {C}(\mathbf{x},\rho)\cap C(\mathring{\mathbf{x}},\pi/2)\right)\mathrm{d}\sigma(\mathbf{x})
    \\ &=\frac{1}{16\pi^2}\cdot 2L(\partial A)\int\limits_{0}^\rho\left(\pi-\alpha(\tau) \cos \rho-2\beta(\tau)\right)\cos \tau\,\mathrm{d}\tau\\ &=\frac{\sin \rho}{16\pi^2}\cdot 2L(\partial A)\int\limits_{0}^1\left(\pi-2\cos\rho\arctan\left(\frac{\sqrt{1-u^2}}{u\cos\rho}\right)-2\arcsin u\right) \mathrm{d}u
    \\&=\frac{1}{8\pi^2}\cdot2L(\partial A)\cdot(\sin \rho-\rho\cos \rho)
    \\&=2L(\partial A)\left(\frac{\rho^3}{24\pi^2}+O(\rho^5)\right).
\end{align*}
The error term is 
\begin{align*}
   \left|M_\rho(A)-I_\rho(A)\right|&=\frac{\rho}{4\pi}\cdot 2L(\partial A)\int\limits_{0}^\rho \frac{\rho^2-\tau^2}{\sin\hat r}\cos \tau\,\mathrm{d}\tau \\&
    =\frac{\rho}{2\pi\sin\hat r}\cdot 2L(\partial A)(\sin\rho -\rho \cos \rho)
    \\
    &\leq\frac{L(\partial A)}{\ \sin \hat r}\cdot \frac{\rho^4}{3\pi}.
\end{align*}
Note that for $0<\rho<\hat{r}/2<\pi/2$,
\begin{align*}
    \left| I_\rho(A)-L(\partial A)\cdot\frac{\rho^3}{24\pi^2} \right|\leq L(\partial A)\cdot \frac{\rho^5}{30} <\frac{L(\partial A)}{\ \sin \hat r}\cdot\frac{\rho^4}{3\pi}.
\end{align*}
\end{proof}

\begin{remark}\label{remark piecewise smooth}
    Note that Lemma \ref{integral M} can be proven for sets $A$ with piecewise smooth boundary. In this case any tube will overlap itself at the corners. The value of $\hat{r}$ can be taken small enough so that there are no overlappings along the sides and between two non-adjacent sides. An additional error term in comparison with the one already present in our formulation of Lemma \ref{integral M} above will come from these regions near corners of $A$, where we have overlappings. These regions are of total size at most $k\sigma\left(C(\rho)\right)$, where $k$ denotes the number of corners. Moreover, we have that $\sigma\left( C\left(\mathbf{x},\rho\right)\backslash A \right)\leq \sigma\left(C(\rho)\right)$. Hence, integrating over these regions as in (\ref{definition_of_M}), we get an additional error term $k\cdot\frac{\rho^4}{16}$.
\end{remark}
Lemma \ref{integral M} shows that the integral in (\ref{The integral}) depends on the boundaries of the sets $A_i$. Namely, the coefficient of the leading term $\rho^3$ in $M_\rho(A)$ depends on the length of $\partial A$, and hence $\mathbb{E}[G_{s,N}]$ depends on $P(\mathcal{A}^{(N)})$, the sum of perimeters of the sets in $\mathcal{A}^{(N)}$.

    Note that for a fixed area $1/N$ of $A$, the smallest possible perimeter is realised by a spherical cap of area $1/N$. Hence the smallest possible coefficient of $\rho^3$ in $M_\rho(A)$ should be for the case of spherical caps. This can be viewed as a corollary of the result of  Feige and Schechtman.

\begin{lemma}[Theorem 5, \cite{Feige2002}]\label{M for caps}
    Fix an $a$ between $0$ and $1$ and $\rho$ between $0$ and $\pi$. Then for the normalized measure $\mu$ on $\mathbb{S}^d$ the maximum of
$$\mu_\rho(A)=\mu^2\left( \{(\mathbf{x},\mathbf{y})\,|\, \mathbf{x}\in A,\, \mathbf{y}\not\in A,\, \vartheta(\mathbf{x},\mathbf{y})\geq\rho\}\right) $$ 
where $A$ ranges over all (measurable) subsets of $\mathbb{S}^d$ of measure $a$ is attained for
a(ny) cap of measure $a$.
\end{lemma}

By Lemma \ref{M for caps}, we obtain a lower bound
\begin{align*}
 \mathbb{E}[G_{s,N}]=N \sum\limits_{i=1}^N M_{sN^{-1/2}}(A_i)\geq N^2 M_{sN^{-1/2}}(C(2\arcsin N^{-1/2}))\approx\frac{s^3}{6\pi}+O(s^4),
\end{align*}
for $N$ large enough.
In order to get stronger ``repulsion'', $P(\mathcal{A}^{(N)})$ has to be as small as possible. However, there is clearly no partition of $\mathbb{S}^2$ into spherical caps. The problem of finding the least-perimeter partition of the sphere into equal-area regions is difficult and largely open. 

Next, we consider a concrete example of an equal-area partition of spheres, the \textit{recursive zonal equal-area partition} $\operatorname{EQ}(d,N)$ of $\mathbb{S}^d$ by Leopardi \cite{Leopardi2006}. The sets of the partition are two polar caps and regions that are rectilinear in spherical polar coordinates.  It was shown in \cite{Leopardi} that $\operatorname{EQ}(d,N)$ is diameter-bounded, i.e. there exists $K_d>0$ such that for all $N\in\mathbb{N}$ and all $R\in\operatorname{EQ}(d,N)$
\begin{align}\label{diameter bounds for EQ}
    \operatorname{diam}(R)\leq K_dN^{-1/d}.
\end{align}
Hence, Proposition \ref{jittered large scale proposition} holds for $\operatorname{EQ}(d,N)$ and $s>K_d$. For small values of $s$, we study the case $d=2$ in more details.

The partition $\operatorname{EQ}(2,N)$ of the sphere $\mathbb{S}^2$ consists of $N$ regions, which form $n$ collars and two polar caps. The number of regions in the $i$th collar is denoted by $m_i$ for $i=1,\dots, n$ ($m_0=m_{n+1}=1$ for polar caps). The angular radius of the polar caps is $$\theta_c=2\arcsin 1/\sqrt{N},$$ so the normalised area is $1/N$. The ideal angle and the ideal number for collars would be $$\delta_I=\sqrt{\frac{4\pi}{N}}\quad \text{and} \quad n_I=\frac{\pi-2\theta_c}{\delta_I}.$$ But the actual number of collars is given by $$n=\operatorname{round}(n_I)\coloneqq\lfloor n_I+1/2\rfloor.$$ The ``fitting'' collar angle is defined as $$\delta_F=\frac{\pi-2\theta_c}{n}=\frac{n_I}{n}\delta_I.$$ It produces ``fitting'' colatitudes of collars: $$\theta_{F,i}=\theta_c+(i-1)\delta_F,$$ where $i=1,\dots,n+1$. The ideal number of regions in the $i$th collar would be 
\begin{align}\label{def y_i}
    y_i=\frac{\sigma(C(\theta_{F,i+1}))-\sigma(C(\theta_{F,i}))}{1/N}=\frac{\cos\theta_{F,i}-\cos\theta_{F,i+1}}{2/N}.
\end{align}
Using a rounding procedure we obtain the actual number of regions in each collar. For $i=1,\dots ,n $, we define 
$$m_i=\operatorname{round}\left(y_i+a_{i-1}\right),\quad\ \ a_i=\sum\limits_{j=1}^i(y_j-m_j), \quad\text{and}\quad   a_0=0.$$
The actual colatitudes can be found by
\begin{align}\label{def theta}
    \sigma(C(\theta_i))=\frac{1}{N}\sum_{j=0}^{i-1}m_j.
\end{align}

Note that the algorithm defines a partition up to rotations of the regions in collars. We can take any partition from this equivalence class, since this will not affect our further computations.
\begin{figure}[h]
    \centering
    \includegraphics[width=0.30\textwidth]{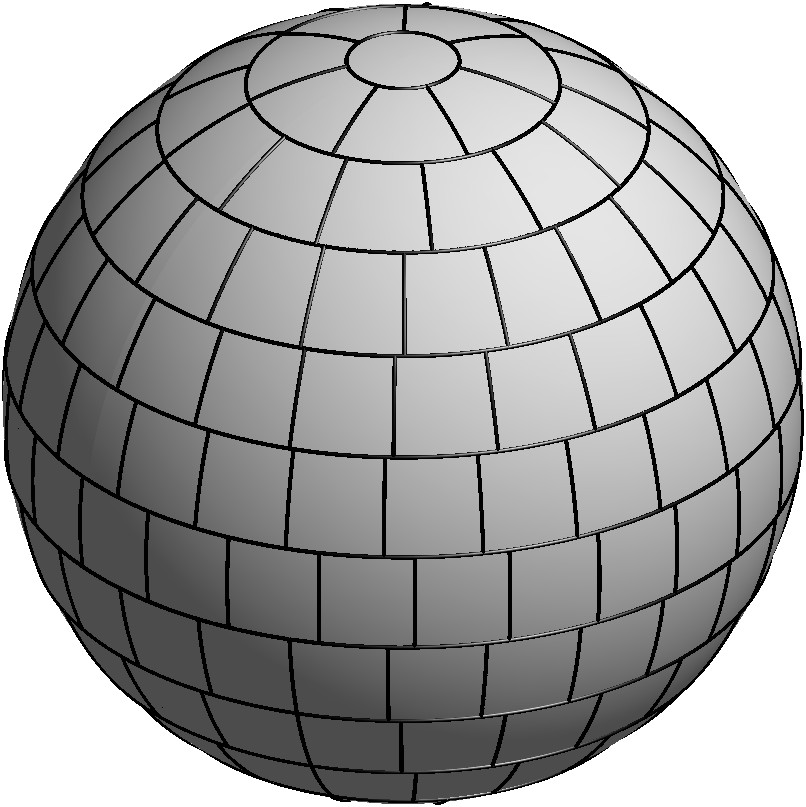}
    \caption{$\operatorname{EQ}(2,200)$.}
    \label{fig:sphere partition}
\end{figure}

Loosely speaking, the regions in $\operatorname{EQ}(2,N)$ look like ``squares'' with side length approximately $\sqrt{4\pi}/\sqrt {N}$. It is expected that the perimeter of each region is close to $8\sqrt{\pi}/\sqrt N$. 
\begin{lemma}\label{total perimeter}
    The sum of the perimeters of all regions in EQ$(2,N)$ is given by 
    $$P(N)=8\sqrt{\pi}\cdot\sqrt{N}+O(1) \quad \text{as}\quad N\rightarrow\infty.$$
\end{lemma}
\begin{proof}
    The union of all the boundaries consists of $n+1$ circles and $N-2$ arcs. The length of the circle at the colatitude $\theta_i$ is equal to $2\pi\sin\theta_i$. The length of one arc on the same level is $\theta_{i+1}-\theta_i$.
    The total length is 
\begin{align}\label{perimeter estimate}
    2\pi\sum\limits_{i=1}^{n+1}\sin\theta_i+\sum\limits_{i=1}^nm_i(\theta_{i+1}-\theta_i).
\end{align}
From the algorithm, it is known that $\theta_1=\theta_{F,1}=\theta_c$ and  $\theta_{n+1}=\theta_{F,n+1}=\pi-\theta_c$. Note that by (\ref{def y_i})  and (\ref{def theta}) we have 
$$\cos\theta_{i}=1-\frac{2}{N}\sum\limits_{j=0}^{i-1}m_j\quad \text{and} \quad \cos\theta_{F,i}=1-\frac{2}{N}\sum\limits_{j=0}^{i-1}y_j.$$
By the rounding procedure for $m_j$'s and $y_j$'s, 
$$\sum\limits_{j=0}^{i-2}y_j<\sum\limits_{j=0}^{i-1}y_j-\frac{1}{2}\leq \sum\limits_{j=0}^{i-1}m_j\leq \sum\limits_{j=0}^{i-1}y_j+\frac{1}{2}<\sum\limits_{j=0}^{i}y_j$$
which gives us 
$$\cos\theta_{F,i+1}<\cos\theta_{F,i}-\frac{1}{N}\leq \cos\theta_i\leq \cos \theta_{F,i}+\frac{1}{N}<\cos \theta_{F,i-1}$$
for $i=2,\dots ,n.$ It can be rewritten as 
$$|\cos\theta_i-\cos\theta_{F,i}|=\left|2\sin\tfrac{\theta_{F,i}+\theta_i}{2}\sin\tfrac{\theta_{F,i}-\theta_i}{2}\right|\leq \frac{1}{N},$$
Note that $\tfrac{\theta_{F,i}+\theta_i}{2}\in(\theta_{F,i-1},\theta_{F,i+1})\subseteq(\theta_c,\pi -\theta_c)$. As $\sin \tfrac{\theta_{F,i}+\theta_i}{2}\neq 0$, this gives us 
$$|\sin\theta_i-\sin\theta_{F,i}|=\left|2\cos\tfrac{\theta_{F,i}+\theta_i}{2}\sin\tfrac{\theta_{F,i}-\theta_i}{2}\right|\leq \frac{1}{N\sin\tfrac{\theta_{F,i}+\theta_i}{2}}\leq\frac{1}{\sqrt N}.$$
We can estimate the first sum in (\ref{perimeter estimate}),
\begin{align*}
    \left|\sum\limits_{i=1}^{n+1}\sin\theta_i-\sum\limits_{i=1}^{n+1}\sin\theta_{F,i}\right| &\leq \sum\limits_{i=2}^n\left|\sin\theta_i-\sin\theta_{F,i}\right|\leq \frac{n-1}{\sqrt N}=O(1) \quad \text{as}\quad N\rightarrow\infty.
\end{align*}
Using summation formulas for $\sin(kx)$ and $\cos(kx)$, we obtain 
\begin{align*}
    \sum\limits_{i=1}^{n+1}\sin\theta_{F,i}&=\sum\limits_{i=1}^{n+1}\sin(\theta_c+(i-1)\delta_F)\\ &=\sin\theta_c+\sin\theta_c\sum\limits_{k=1}^n \cos (k\delta_F)+\cos\theta_c\sum\limits_{k=1}^n\sin(k\delta_F)
    \\ &= \sin\theta_c+\sin\theta_c\cdot \frac{\sin\frac{n\delta_F}{2}\cos\frac{(n+1)\delta_F}{2}}{\sin\frac{\delta_F}{2}}+\cos\theta_c\cdot \frac{\sin\frac{n\delta_F}{2}\sin\frac{(n+1)\delta_F}{2}}{\sin\frac{\delta_F}{2}}.
\end{align*}
Using the fact that $n\delta_F=\pi-2\theta_c$, we get a simplified expression for the sum,
\begin{align*}
      \sum\limits_{i=1}^{n+1}\sin\theta_{F,i}&=\sin\theta_c+\sin\theta_c\cdot \frac{\cos\theta_c\sin(\theta_c-\frac{\delta_F}{2})}{\sin\frac{\delta_F}{2}}+\cos\theta_c\cdot \frac{\cos\theta_c\cos(\theta_c-\frac{\delta_F}{2})}{\sin\frac{\delta_F}{2}}
      \\ &=\sin\theta_c+\cos\theta_c\cot\frac{\delta_F}{2}
      \\ &=\sqrt{\frac{N}{\pi}}+O(1/\sqrt{N})\quad \text{as}\quad N\rightarrow\infty.
\end{align*}
For the second sum in (\ref{perimeter estimate}) we have
\begin{align*}
    \sum\limits_{i=1}^nm_i(\theta_{i+1}-\theta_i)&=\sum_{i=1}^nm_i-(\theta_{F,i+1}-\theta_{F,i}+\alpha_{i+1}-\alpha_i)
    \\&=\sum\limits_{i=1}^nm_i\delta_F+\sum\limits_{i=1}^nm_i(\alpha_{i+1}-\alpha_i)
    \\&=S_{1}+S_2,
\end{align*}
where $\alpha_i=\theta_i-\theta_{F,i}$. The main term is
$$S_1=\delta_F\sum\limits_{i=1}^nm_i=\left(\delta_I+O(1/N)\right)(N-2)=2\sqrt{\pi}\cdot\sqrt{N}+O(1).$$
Recall that $\alpha_1=\alpha_{n+1}=0$ and
\begin{equation*}
    m_{i}=\frac{\cos\theta_i-\cos\theta_{i+1}}{2/N}=N\sin\tfrac{\theta_i+\theta_{i+1}}{2}\sin\tfrac{\theta_{i+1}-\theta_{i}}{2}\leq2\sqrt N\sin \tfrac{\theta_i+\theta_{i+1}}{2}.
\end{equation*}
It implies that 
\begin{equation*}
    |m_{i}-m_{i-1}|\leq 4\sqrt{N}\sin\left(\tfrac{\theta_i}{2}+\tfrac{\theta_{i-1}+\theta_{i+1}}{4}\right).
\end{equation*}
The second sum is
\begin{align*}
    |S_2|=\left|\sum_{i=2}^n(m_i-m_{i-1})\alpha_i\right|&\leq\frac{1}{\sqrt{N}} \sum\limits_{i=2}^n\frac{\sin\left(\tfrac{\theta_i}{2}+\tfrac{\theta_{i-1}+\theta_{i+1}}{4}\right)}{\sin\tfrac{\theta_{i}+\theta_{F,i}}{2}}
    \\&\leq \frac{1}{\sqrt{N}} \sum\limits_{i=2}^n\left(1+\tfrac{1}{\sqrt{N}}\left|\cot\tfrac{\theta_{i}+\theta_{F,i}}{2}\right|\right)
    \\ &=\frac{1}{N}\sum\limits_{i=2}^n\left|\cot\tfrac{\theta_{i}+\theta_{F,i}}{2}\right|+O(1)\quad \text{as}\quad N\rightarrow\infty.
\end{align*}
Applying Koksma's inequality for $n-1$ points $x_i=\frac{\theta_{F,i}+\theta_i}{2}$ on $(\theta_c,\pi-\theta_c)$ and the function $f(x)=|\cot x|$, we estimate the error term by
\begin{align*}
   |S_2| &\leq \frac{n-1}{N}\int\limits_{\theta_c}^{\pi-\theta_c}|\cot(x)|\mathrm{d}x+\frac{n-1}{N}\cdot 2\cot \theta_c\cdot \frac{2}{n-1}+O(1)
    \\ &\leq\frac{1}{\sqrt{N}}\cdot 2\log\left(\frac{1}{\sin\theta_c}\right)+\frac{4\cot \theta_c}{N}+O(1)
    \\ & =O\left(1\right)\quad \text{as}\quad N\rightarrow\infty.
\end{align*}
Combining all together we get the desired result.
\end{proof}

\begin{prop}
 For points $\mathbf{x}_1,\dots,\mathbf{x}_N$ in $\mathbb{S}^2$ generated by the jittered sampling associated to $\operatorname{EQ}(2,N)$, we have
\begin{equation*}
     \lim\limits_{N\rightarrow\infty}\mathbb{E}[G_{s,N}(\mathbf{x}_1,\dots,\mathbf{x}_N)]=\frac{s^2}{4}-1\quad\text{for}\quad s>K_2
\end{equation*}
and for a fixed $0<s<1/4$ and large $N$
\begin{equation*}
   \left|\mathbb{E}[G_{s,N}]-\frac{s^3}{8\pi^2}\right|\leq C_2s^4\,
\end{equation*}
for some $C_2>0.$
\end{prop}
   
\begin{proof}
By Proposition \ref{jittered large scale proposition} and (\ref{diameter bounds for EQ}), we get the asymptotics for large $s$. The value of $K_2$ is approximately 12.8. When $s$ is small, we can apply Lemma \ref{integral M} and Remark \ref{remark piecewise smooth} with $\hat{r}=\frac{1}{2\sqrt{N}}$ and $k=4$ to the regions in $\operatorname{EQ}(2,N)$. The boundaries of the regions consist of arcs of great circles and smaller circles (that form collars of the partition). For great circles, there is no overlapping for $0<\hat{r}<\pi/2$. For a circle of radius $r$ (as it is defined in (\ref{no overlap condition})), we can take $0<\hat{r}<r$. The smallest radius have the polar circles,
$$r_{p}=\sin\left(2\arcsin\left(1/\sqrt{N}\right)\right)>1/\sqrt{N}\quad \text{for} \quad N>1.$$

So we can take $\hat{r}=\frac{1}{2\sqrt{N}}$ and there will be no overlapping along the arcs and with opposite sides of ``rectangular'' regions.

Let $A_1^{(N)},\dots ,A_N^{(N)}$ be the regions of $\operatorname{EQ}(2,N)$. For $0<s<1/4$ and $\rho=s/\sqrt{N}$, we get that 
\begin{align*}
     \left|N\sum\limits_{i=1}^N M_{\rho}\left(A_i^{(N)}\right)-NP(N)\frac{\rho^3}{24\pi^2}\right|\leq NP(N)\hat{r}^{-1}\rho^4 +N^2\frac{\rho^4}4
\end{align*}
Using Lemma \ref{total perimeter}, we have
\begin{align*}
      \left|\mathbb{E}[G_{s,N}]-\frac{s^3}{8\pi^2}\right|\leq 16\sqrt{\pi}s^4+\frac{s^4}{4}+\frac{Cs^3}{24\pi^2\sqrt{N}}+\frac{16\sqrt{\pi}Cs^4}{\sqrt{N}}
\end{align*}
for some $C>0$ independent of $s$ and $N$. Taking $N$ large enough completes the proof.
\end{proof}
Accordingly, in this model there is ``repulsion'' for small $s$, but less in comparison to the harmonic ensemble and the spherical ensemble. For large $s$, the behaviour of $\mathbb{E}[G_{s,N}]$ matches the asymptotics of the i.i.d.\ case.

\section*{Acknowledgments}

The author is supported by the Austrian Science Fund (FWF), project DOC-183. 
%\bibliographystyle{abbrv}
%\bibliography{biblio}

\end{document}